\documentclass[11pt]{article}
\usepackage[cp1250]{inputenc}
\usepackage{amsmath}
\usepackage{amssymb}
\usepackage{amsfonts}
\usepackage{amsthm}
\usepackage{cite}
\usepackage{graphicx} 
\usepackage{booktabs} 
\usepackage{color}
\usepackage{enumitem}

\usepackage{latexsym}

\usepackage{geometry}
\usepackage{hyperref} 

\newtheorem{theorem}{Theorem}[section]
\newtheorem{corollary}[theorem]{Corollary}
\newtheorem{lemma}[theorem]{Lemma}
\newtheorem{proposition}[theorem]{Proposition}

\theoremstyle{definition}
\newtheorem{definition}[theorem]{Definition}

\newtheorem{question}[theorem]{Question}
\newtheorem{example}[theorem]{Example}
\newtheorem{remark}[theorem]{Remark}

\newcommand*{\cdef}{\newcommand*}

\let \P \undefined

\cdef \A {\mathcal{A}}
\cdef \F {\mathcal{F}}
\cdef \P {\mathcal{P}}
\cdef \T {\mathcal{T}}
\cdef \bbn {\mathbb{N}}
\cdef \bbr {\mathbb{R}}
\cdef \eps {\varepsilon}

\cdef \ord {\operatorname{ord}} 
\cdef \CPM {\mathcal{CMM}} 
\cdef \card {\operatorname{card}} 
\cdef \dia {\operatorname{diam}} 
\cdef \HD {\operatorname{HD}} 
\cdef \Lip {\operatorname{Lip}} 
\cdef \Per {\operatorname{Per}} 
\cdef \Var {\operatorname{Var}} 
\cdef \htop {h_\textup{top}} 

\cdef \maps {\colon} 

\definecolor{gray}{rgb}{0.4, 0.4, 0.4}

\setlist{itemsep = 0pt}

\cdef \DescriptionFormat [1]{%
	\normalfont\emph{#1}%
}

\setlist[description]{format=\DescriptionFormat}

\hypersetup{
	pdftitle = {Constant slope, entropy and horseshoes for a map on a tame graph},
	pdfauthor = {A. Barto\v s, J. Bobok, P. Pyrih, S. Roth, B. Vejnar},
	linktoc = page,
	linkbordercolor = {0.5 0.5 1},
	citebordercolor = {0.5 1 0.5}
}

\begin{document}

\title{Constant slope, entropy and horseshoes \\ for a map on a tame graph}

\author{Adam Barto\v s\footnote{The author was supported by the grants GAUK 970217 and SVV-2017-260456 of Charles University.}\\
Charles University, 
Faculty of Mathematics and Physics,\\
Department of Mathematical Analysis\\
e-mail: \texttt{drekin@gmail.com}\\
\\
Jozef Bobok\\
Czech Technical University in Prague,\\  Faculty of Civil Engineering\\
\\
Pavel Pyrih\\
Charles University, 
Faculty of Mathematics and Physics,\\
Department of Mathematical Analysis\\
\\
Samuel Roth\\
Silesian University in Opava,\\ Mathematical Institute\\
\\
Benjamin Vejnar\footnote{The author was supported by the grant GA\v CR 17-04197Y and Charles University Research Centre program No.~UNCE/SCI/022.
}
\\
Charles University, 
Faculty of Mathematics and Physics,\\
Department of Mathematical Analysis
}

\date{May 3, 2018}

\maketitle

\begin{abstract}
	We study continuous countably (strictly) monotone maps defined on a tame graph, i.e., a special Peano continuum for which the set containing branchpoints and endpoints has a countable closure. In our investigation we confine ourselves to the countable Markov case. We show a necessary and sufficient condition under which a locally eventually onto, countably Markov map $f$ of a tame graph $G$ is conjugate to a constant slope map $g$ of a countably affine tame graph. In particular, we show that in the case of a Markov map $f$ that corresponds to recurrent transition matrix, the condition is satisfied for constant slope $e^{\htop(f)}$, where $\htop(f)$ is the topological entropy of $f$. Moreover, we show that in our class the topological entropy $\htop(f)$ is achievable through horseshoes of the map $f$.
	
	\begin{description}
		\item[Classification:]
			Primary
			37E25; 
			Secondary
			37B40, 
			37B45. 
		
		\item[Keywords:] countably Markov map, tame graph, constant slope, mixing, locally eventually onto, conjugacy, entropy, horseshoes.
	\end{description}
\end{abstract}

\bigskip

\tableofcontents
\bigskip

\linespread{1.2}\selectfont 

\section{Introduction}\label{sec:Intro}

In this paper we explore when a \emph{countably (strictly) monotone} map on a tame graph is conjugate to a map of constant slope.
We show a necessary and sufficient condition under which a countably Markov and locally eventually onto map $f$ of a tame graph $G$ is conjugate to a constant slope map $g$ of a countably affine tame graph. In particular, we show that in the case of a Markov map $f$ that corresponds to a recurrent transition matrix, the condition is satisfied for constant slope $e^{\htop(f)}$, where $\htop(f)$ is the topological entropy of $f$. We give also a partial solution to the constant slope problem when the locally eventually onto hypothesis is weakened to topological mixing. Moreover, we show that in this broader class of maps the topological entropy $\htop(f)$ is achievable through horseshoes of the map $f$.

Let us consider continuous maps $f\maps X \to X$, $g\maps Y\to Y$ and $\phi\maps X\to Y$, where
$X$, $Y$ are compact Hausdorff spaces, such that
\begin{equation}\label{e:41}\phi\circ f=g\circ \phi.\end{equation}
If $\phi$ is surjective/homeomorphism, we say that $f$ is \emph{semiconjugate}/\emph{conjugate} to $g$ via the map $\phi$. We call the map $\phi$ a \emph{semiconjugacy}/\emph{conjugacy}.

Recall that a \emph{finite graph} is a Peano continuum which can be written as the union of finitely many arcs any two of which are either disjoint or intersect only in one or both of their endpoints, and that a \emph{finite tree} is a finite graph not containing a simple closed curve.

Let $G$ be a finite graph. A continuous map $f\maps G\to G$ is said to be \emph{piecewise strictly monotone} if there is a finite subset $C_f \subset G$ such that for each connected component $P$ of $G \setminus C_f$, $f\vert_{P}$ is a homeomorphism of $P$ onto its image.

The classical results on conjugacy of a piecewise monotone interval map to a map of constant slope come from Parry \cite{Par66} and later on Milnor and Thurston \cite{MiThu88} (see also \cite{BoBru15}). A tree version concerning piecewise strictly monotone tree maps had been proved by Baillif and de Carvalho in \cite{BaCar01}. Recently, Alseda and Misiurewicz proved, among other results, its graph version:

\begin{theorem}\cite{AlMi15}
Let $f\maps G\to G$ be a piecewise monotone map of a finite graph $G$ with positive topological entropy $\htop(f)$. Then, there is a semiconjugacy $\phi$ between $f$ and a piecewise monotone graph map $g\maps G'\to G'$ with constant slope $e^{\htop(f)}$. The map $\phi$ is a homeomorphism if $f$ is transitive.
\end{theorem}

Throughout the paper we will use topological properties of special continua -- as a general reference with an extensive exposition concerning Peano continua see \cite{Nad92}.

\subsection{Tame graphs}

We will consider the class of so called \emph{tame graphs}, a particular class of ``countable graphs'' generalizing finite graphs.

Following \cite[Definition 9.3]{Nad92}, by $\ord(x, X)$ we denote the \emph{order of a point} $x$ in a space $X$. We denote by $E(X)$, resp. $B(X)$ the set of all points $x\in X$ such that $\ord(x, X)=1$ (\emph{endpoint}), resp. $\ord(x, X)\geq 3$ (\emph{branchpoint}).
An arc $\alpha$ in a continuum $X$ is called a \emph{free arc} if the set $\alpha^\circ = \alpha \setminus E(\alpha)$ is open in $X$.

A continuum $G$ will be called a \emph{tame graph} if the set $E(G)\cup B(G)$ has a countable closure.
A \emph{tame partition} for $G$ is a countable family $\P$ of free arcs with pairwise disjoint interiors covering $G$ up to a countable set of points.

\begin{figure}[htb!!]
	\centering
	\bigskip
	\includegraphics[width=11cm]{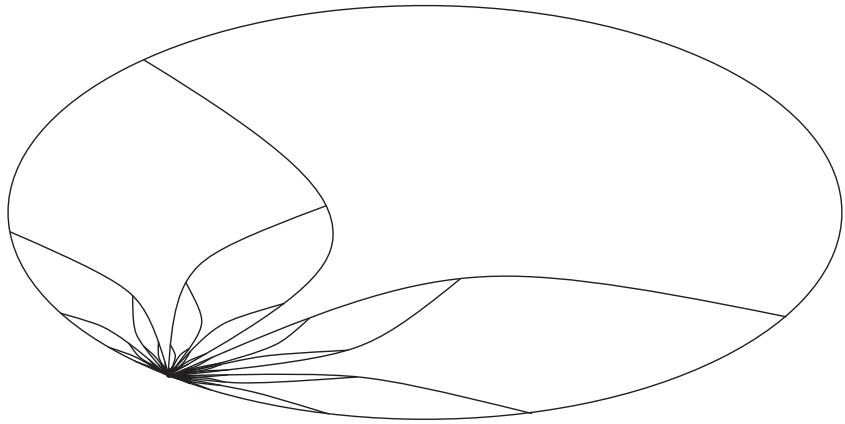}
	\caption{A planar tame graph that is a quotient of the Gehman dendrite.}
\end{figure}

\begin{theorem}
	The following conditions are equivalent for a continuum $G$:
	\begin{itemize}
		\item[(i)] $G$ is a tame graph.
		\item[(ii)] $G$ is the union of a countable sequence of free arcs and a countable set.
		\item[(iii)] $G$ admits a tame partition.
		\item[(iv)] Each point of $G$ but countably many has a neighborhood that is a finite graph.
	\end{itemize}
\end{theorem}

\begin{proof}
If $G$ is degenerate or a simple closed curve, the theorem follows easily.

(i)$\implies$(ii):~By \cite[Theorem 10.4]{Nad92} one obtains that every tame graph is \emph{hereditarily locally connected} (otherwise there would be so called \emph{continuum of convergence} and hence uncountable many points of infinite order).

Let $F$ be the countable closure of $E(G)\cup B(G)$. For every point $x\in G\setminus F$ there is a closed and connected neighborhood $A$ of $x$ which does not intersect $F$.
Clearly every point in $A$ is of order at most two. It follows by \cite[Proposition 9.5]{Nad92} that $A$ is either an arc or a simple closed curve. Since we assume that $G$ is neither degenerate nor a simple closed curve, $A$ has to be an arc. It follows that $x$ is an element of a free arc. Now, we are able to find countably many free arcs covering $G\setminus F$.


(ii)$\implies$(iii) follows from the fact that every free arc of $G$ is contained in a maximal free arc or in a free loop (i.e. a simple closed curve $\alpha \subset G$ such that $\alpha \setminus \{x\}$ is open in $G$ for some point $x \in \alpha$). These maximal free arcs and free loops have pairwise disjoint interiors.

(iii)$\implies$(iv):~Clearly every point in the interior of a free arc has this free arc as a neighborhood.

(iv)$\implies$(i):~No point of $G$ that has a finite graph neighborhood can be a limit point of $E(G) \cup B(G)$.
\end{proof}

\begin{remark}
	We have seen in the proof that every tame graph is hereditarily locally connected. It is an easy exercise to verify that every subcontinuum of a tame graph is a tame graph as well.
\end{remark}

\begin{remark}
	Every finite graph is a tame graph. Also, being a finite graph is equivalent to any of the following conditions:
		(i) $E(G) \cup B(G)$ is finite;
		(ii) $G$ is the finite unions of free arcs;
		(iii) $G$ admits a finite tame partition;
		(iv) each point of $G$ has a finite graph neighborhood.
\end{remark}

By the above theorem and the sum theorem \cite[Theorem 1.5.3]{En78}, every tame graph is a one-dimensional Peano continuum. Any tame graph $G$ is considered to be embedded in 3-dimensional Euclidean space~\cite{CoEaLiRu95}, from which it inherits the Euclidean metric $d$ inducing its topology $\tau$.

Euclidean space also carries a function $\ell$ which assigns to each rectifiable curve its length. Unfortunately, we cannot define constant slope maps on arbitrary tame graphs since the arcs composing the graph might not be rectifiable. Therefore we proceed as follows.

A tame graph $G'\subset\mathbb{R}^3$ is called \emph{countably affine} if it admits a tame partition whose arcs are all line segments.
The \emph{total length} of $G'$ is the sum of the lengths of those line segments, and may be either finite or infinite.
We say that $g\maps G' \to G'$ is a \emph{countably affine graph map} if $G'$ admits a tame partition into segments such that the restriction of $g$ to each segment is an affine map.
Moreover, every affine map between two line segments has a naturally defined \emph{slope}, namely, the ratio of the lengths of the line segments, and $g$ is said to have \emph{constant (resp. bounded) slope $\lambda$} if the restrictions of $g$ have slope $\lambda$ (resp. $\leq\lambda$).

\subsection{Main goal and structure of exposition}

The question we want to address is: \emph{when is a continuous countably Markov and mixing map $f$ of a tame graph $G$ conjugate to a countably affine graph map $g\maps G'\to G'$ of constant slope $e^{\htop(f)}$?} Our general strategy will follow that of~\cite{BoRo17} with some modification to adapt it to more general tame graphs.

In Section~\ref{s:markov} we define our class of maps and review the basic properties of Markov partitions. Section~\ref{s:main} develops our main results, giving conditions for the existence of a conjugate map of constant slope. Section~\ref{s:horseshoes} makes a stronger connection between a Markov tame graph map and its symbolic dynamics, allowing us to show that entropy in our class of maps is given by horseshoes. Finally, Section~\ref{s:lipschitz} relates the entropy of the maps we consider to the Lipschitz constants of compatible metrics.

\section{Markov maps on tame graphs}\label{s:markov}

Let $f\maps G \to G$ be a continuous map on a tame graph.
A family $\P$ is called a \emph{countable Markov partition} for $f$ if
\begin{itemize}
	\item $\P$ is a tame partition for $G$. In particular, for every $i \in \P$, $i^\circ = i \setminus E(i)$ is open in $G$, and so any point in $i \cap B(G)$ is an endpoint of $i$.
	\item For every $i \in \P$, $f|_i$ is \emph{monotone}, i.e. $f$ maps $i$ homeomorphically onto its image $f(i)$. So the partition $\P$ witnesses that the map $f$ is \emph{countably monotone}.
	\item For every $i , j \in \P$, if $f(i) \cap j^{\circ} \neq \emptyset$, then $f(i) \supset j$.
\end{itemize}

\begin{remark}\label{r:6}Let $G$ be a tame graph, consider a countable Markov partition $\P$ for a map $f\maps G\to G$, and let $\alpha\subset G$ be a simple closed curve. Then the set $P\cap\alpha$ contains at least two points.
\end{remark}

A continuous map $f\maps G\to G$ is said to belong to the class $\CPM(G)$ (\emph{countably Markov and mixing}) if
\begin{itemize}
\item $f$ admits a countably infinite Markov partition.
\item $f$ is \emph{topologically mixing}, i.e., for every pair of nonempty open sets $U,V$ there is an $n$ such that $f^m(U)\cap V\neq\emptyset$ for all $m\ge n$.
\end{itemize}

\begin{remark}\label{r:1} We will argue in Corollary \ref{c:3} that for each $f\in\CPM(G)$, the set $\Per(f)$ of periodic points is dense in $G$ and $\htop(f)>0$.
\end{remark}

\begin{remark}\label{r:2}~A map $f$ from $\CPM(G)$ has to satisfy one of the following two possibilities: Either there exists a nonempty open set $U$ such that for every $n$, $f^n(U)\subsetneq G$; or for every nonempty open set $U$ there is an $n$ for which $f^n(U)=G$. In the latter case, $f$ will be called \emph{leo} (\emph{locally eventually onto}). Our attention will be paid to both leo/non-leo types of maps.
\end{remark}

\begin{remark}\label{r:4}
	Let $G$ be a finite graph and let $f\maps G\to G$ be a continuous piecewise monotone Markov graph map, i.e., such that $f$ admits a finite Markov partition. If $f$ is topologically mixing, then $f$ also admits countably infinite Markov partitions, and we will consider $f$ as an element of $\CPM(G)$.
\end{remark}

The basic properties of Markov partitions as regards iteration are summarized in the next lemma.

\begin{lemma}\label{lem:refine}Suppose $f\in \CPM(G)$ with a partition set $\P$, let $n\ge 0$.
  Then
\begin{enumerate}[label=(\roman*)]
\item\label{it:mar} $P_n:=f^{-n}(P)$  defines a countable Markov partition $\P_n$ for $f^j$ whose elements are the closures of connected components of $G\setminus P_n$, provided $j\leq n+1$.
\item\label{it:int} The partition arcs of $\P_n$ are the arcs of the form
\begin{equation*}
[i_0i_1\cdots i_n]:=i_0\cap f^{-1}i_1 \cap \cdots \cap f^{-n}i_n, \quad \text{where } i_0,\ldots,i_n \in\P.
\end{equation*}
\item\label{it:hom} For each $[i_0\cdots i_n]\in\P_n$, the restricted map $f^n|_{[i_0\cdots i_n]}\maps [i_0\cdots i_n] \to [i_n]$ is a homeomorphism.
\item\label{it:dense} $Q:=\bigcup_{n=0}^\infty P_n = \bigcup_{i=0}^\infty f^{-i}(P)$ is dense in $G$.
\end{enumerate}
\end{lemma}

We omit the proof of Lemma \ref{lem:refine}. For a given $f\in\CPM(G)$ with a Markov partition $\P$ we associate to $f$ and $\P$ the \emph{transition matrix} $M=M(f) = (m_{ij})_{i,j\in \P}$ defined by
\begin{equation}\label{e:3}
m_{ij} = \begin{cases}
1 & \text{if } f(i) \supset j,\\
0 & \text{otherwise.}
\end{cases}
\end{equation}
This matrix in turn determines a countable state Markov shift $(\Sigma_M,\sigma)$, where
\begin{equation*}
\Sigma_M=\left\{i_0i_1\cdots\in\P^{\mathbb{N}_0}\colon~m_{i_n i_{n+1}}=1, n=0,1,\ldots\right\}
\end{equation*}
and $\sigma(i_0i_1\cdots)=i_1i_2\cdots$ is the shift map. The topology of $\Sigma_M$ is inherited from the product space $\P^{\mathbb{N}_0}$, where $\P$ has the discrete topology. The partition arcs of $\P_n$ described in Lemma~\ref{lem:refine}~\ref{it:int} correspond to all the words $i_0\cdots i_n$ which are allowable in this shift space, i.e., such that $m_{i_0i_1}=\cdots=m_{i_{n-1}i_n}=1$.


%

\section{Conjugacy to maps of constant or bounded slope}\label{s:main}

Let $f\in\CPM(G)$ with a Markov partition $\P$ and its transition matrix $M(f)=(m_{ij})_{i,j\in \P}$. We will be interested in positive real numbers $\lambda>1$ and positive sequences $(v_i)_{i\in\P}$ satisfying the inequalities

\begin{equation}\label{e:2}\forall~i\in \P\colon~\sum_{j\in \P}m_{ij}v_j\le \lambda~v_i.\end{equation}

Any nonzero nonnegative sequence $v=(v_i)_{i\in\P}$ satisfying (\ref{e:2}) will be called a \emph{$\lambda$-subeigenvector}. We call $v$ \emph{summable} if $v\in\ell^1(\P)$.
We say $v$ is \emph{deficient in coordinate $i$} if 
there is strict inequality $\sum_{j\in \P} m_{ij} v_j < \lambda v_i$.

We call $v$ a \emph{$\lambda$-eigenvector} if there is no deficiency, i.e. if there is equality in~\eqref{e:2} for each $i\in\P$.

\begin{remark}\label{r:5}
    (i) The mixing property of $f$ immediately implies \emph{irreducibility} of $M$ (for all $i,j\in\P$ there is $n\in\mathbb{N}$ such that $m_{ij}(n)>0$).  Second, irreducibility of $M$ immediately implies that each subeigenvector has all entries strictly positive. (ii) If $f$ is leo, then every $\lambda$-subeigenvector is summable. This is because (\ref{e:2}) can be rewritten for $n>0$ as
    $$\forall~i\in \P\colon~\sum_{j\in \P}m_{ij}(n)v_j\le \lambda^n~v_i$$
    where $M^n=(m_{ij}(n))_{i,j\in\P}$ -- see (\ref{e:11}) -- and for $f$ leo,
    $$\forall~i\in\P~\exists~n_0(i)~\forall~n>n_0(i)~\forall~j\in\P\colon~ m_{ij}(n)\ge 1,$$
    hence for a given $i$ and $n>n_0(i)$
    $$\sum_{j\in\P}v_j\le \sum_{j\in\P}m_{ij}(n)v_j\le \lambda^n~v_i<\infty.$$
\end{remark}

Using Remark \ref{r:5}(i), for a $\lambda$-subeigenvector $v$ and all $i\in\P$ we put

\begin{equation}\label{deltapsi}
\lambda_{i}:=\frac{(Mv)_i}{v_i}\le \lambda.
\end{equation}

\subsection{Statement of main results}

We now state our main results. We give a partial solution to the constant slope problem for mixing maps, and a complete solution for leo maps.
Since we will work with maps from $\CPM(G)$ that are topologically mixing, we formulate our statement for topological conjugacies only -- compare with \cite[Proposition 4.6.9]{ALM00}.

\begin{theorem}\label{t:main}
Let $f\in\CPM(G)$ with transition matrix $M$. In order for $f$ to be conjugate to a countably affine graph map of constant slope $\lambda>1$,
\begin{enumerate}
\item[(i)] it is necessary that $M$ admits a $\lambda$-eigenvector, and
\item[(ii)] it is sufficient that $M$ admits a summable $\lambda$-eigenvector.
\end{enumerate}
\end{theorem}

By remark~\ref{r:5}~(ii), this gives us our solution for leo maps.

\begin{corollary}\label{c:main}
Let $f\in\CPM(G)$ be leo with transition matrix $M$. Then $f$ is conjugate to a countably affine graph map of constant slope $\lambda>1$ if and only if $M$ admits a $\lambda$-eigenvector.
\end{corollary}

\begin{remark}The condition of Corollary \ref{c:main} with $\lambda=e^{\htop(f)}$ is fulfilled when the map $f$ is leo and its transition matrix $M(f)$ is recurrent -- compare with \cite{BoBru15}.
\end{remark}

Our proof techniques apply not only to constant slope maps, but also to bounded slope maps, leading us to the following theorem whose consequences we explore in Section~\ref{s:lipschitz}. The interval version of the following theorem was recently proved in \cite[Theorem 3.7]{BoRo17}.

\begin{theorem}\label{t:broad}
Let $f\in\CPM(G)$ with transition matrix $M$. In order for $f$ to be conjugate to a countably affine graph map of bounded slope $\lambda>1$,
\begin{enumerate}
\item[(i)] it is necessary that $M$ admits a $\lambda$-subeigenvector, and
\item[(ii)] it is sufficient that $M$ admits a summable $\lambda$-subeigenvector with deficiency in only finitely many entries.
\end{enumerate}
\end{theorem}

The necessity of the (sub)eigenvector condition is more or less trivial to prove.
\begin{proof}[Proof of Theorem~\ref{t:main}~(i) and Theorem~\ref{t:broad}~(i)]
Write $\P$ for the Markov partition, $g\maps G'\to G'$ for the countably affine graph map of constant (bounded) slope $\lambda$, and $\phi\maps G\to G'$ for the conjugating homeomorphism. Now put $v_i=\ell(\phi(i))$ for all $i\in\P$. The constant (bounded) slope condition implies immediately that $v$ is a $\lambda$-(sub)eigenvector.
\end{proof}

The sufficiency of the (sub)eigenvector condition is much harder to prove. It requires us to construct the countably affine graph $G'$, the constant (bounded) slope map $g$, and the conjugating homeomorphism $\phi\maps G\to G'$ all from scratch. To do it, we first construct an \emph{arc map} $A$ for $G$ with the same properties as we want the length function $\ell$ to have with respect to $G'$, $g$.

\subsection{Arc maps}

In what follows we assume that a summable $\lambda$-subeigenvector  $v=(v_i)_{i\in \P}$ satisfying (\ref{e:2}) is normalized so that $\sum v_i=1$, and that the set $\P'=\{i:~\lambda_i<\lambda\}$ of coordinates in which $v$ is deficient is finite. Denote for $[i_0\cdots i_n]\in\P_n$,
\begin{equation}\label{e:8}\Delta([i_0\cdots i_n]):=
\frac{v_{i_n}}{\prod_{\ell=0}^{n-1}\lambda_{i_{\ell}}},
\end{equation}
where the empty product (when $n=0$) is taken to be $1$; by Remark \ref{r:5} $$\Delta([i_0\cdots i_n])>0\text{ for each }n\ge 0.$$

For a Peano continuum, let us denote by $\A(X)$ the set containing all points from $X$ and all arcs in $X$. Any map $F\maps \A(X)\to\bbr$ is called an arc-map defined on $X$. Using (\ref{e:8}) we define arc-maps $A, A_n\maps \A(G)\to [0,\infty)$ by
\begin{equation}\label{e:6}
	A_n(\gamma) = \sum_{\gamma \supset [i_0\cdots i_n] \in \P_n} \Delta([i_0\cdots i_n])
\end{equation}
and $A_n(\gamma)=0$ when $[i_0\cdots i_n]\nsubseteq \gamma$ for any $[i_0\cdots i_n]\in\P_n$. Finally we put
\begin{equation}\label{e:22}
A(\gamma)=\sup_{n\in\mathbb N}\quad A_n(\gamma)
\end{equation}
for every $\gamma\in\A(G)$.  For an arc $\gamma\in\A(G)$ and a point $t\in\gamma$ we denote by $\gamma_t,\gamma_{1-t}$ the elements from $\A(G)$ satisfying
 $$\gamma=\gamma_t\cup \gamma_{1-t},~\{t\}=\gamma_t\cap\gamma_{1-t}.$$
For distinct points $x,y\in G$ we denote by $\A_{xy}$ the set of all arcs in $G$ with endpoints $x,y$.

\begin{proposition}\label{p:1}Let $A_n,A$ be the arc-maps defined in (\ref{e:6}), (\ref{e:22}) and $\gamma\in\A(G)$. Then:
\begin{itemize}
\item[(i)] $A_n(\gamma)\leq A_{n+1}(\gamma)$ hence $A(\gamma)=\lim_{n\to\infty}A_n(\gamma)$.
\item[(ii)] If $\gamma$ is not a point then  $A(\gamma)>0$.
\item[(iii)] The arc-map $A$ is additive (on $\gamma$), i.e., for every $t\in \gamma$, $A(\gamma)= A(\gamma_t)+A(\gamma_{1-t})$.
\item[(iv)] For every sequence $(\gamma^m)_m\subset \A(G)$, if $\gamma^m\supset\gamma^{m+1}$ and $\bigcap_{m}\gamma^m=\{x\}$ then $\lim_{m}A(\gamma^m)=0$.
\item[(v)]  For each pair of distinct points $x,y\in G$ there exist an arc $\alpha\in\A_{xy}$ such that
\begin{equation}\label{e:24}A(\alpha)=\inf\{A(\beta)\colon~\beta\in \A_{xy}\}.\end{equation}
\end{itemize}
\end{proposition}
\begin{proof}
In the proof we repeatedly use some topological properties of Peano continua. All of them can be found in \cite{Nad92}.  We will also often apply the conclusions of Lemma \ref{lem:refine}.

As before let us denote $P=G\setminus \bigcup_{i\in\P}i^{\circ}$, $P_n=f^{-n}(P)$, $Q=\bigcup_{n\ge 0}P_n$. Clearly $f^{-1}(Q)=Q$ and by our choice of $\P$ the set $Q$ is countable and dense in $G$. The set $\P_n$ contains the closures of connected components of $G\setminus P_n$. In particular, $\P_0=\P$.   Also (see \cite[(3),~(4)]{BoRo17} for the details)

\begin{equation}\label{welldefinedDelta}
\sum_{j\subseteq f(i_n)}\Delta([i_0\cdots i_n j]) = \Delta([i_0\cdots i_n]),
\end{equation}
and
\begin{equation}\label{expansionDelta}
\Delta([i_1\cdots i_n])=\lambda_{i_0}\cdot\Delta([i_0\cdots i_n]).
\end{equation}

Notice that $P_n\subset P_{n+1}$ and that each $[i_0\cdots i_n]\in\P_n$ is subdivided by the points of $P_{n+1}$ into the arcs $[i_0\cdots i_n j]$, where $j$ ranges over all members of $\P$ contained in $f(i_n)$.   This fact together with (\ref{e:8}) imply that for each $n$,
\begin{equation}\label{e:5}
	\sum_{[i_0\cdots i_n]\in\P_n}\Delta([i_0\cdots i_n])=\sum_{i\in\P}\Delta([i])=\sum_{i\in\P}v_i=1.
\end{equation}

It follows from (\ref{e:6}) and \eqref{welldefinedDelta} that for every $\gamma\in\P_n$,
$A_m(\gamma)=A_n(\gamma)$ whenever $m\ge n$.
Moreover, again from \eqref{welldefinedDelta} we obtain for any arc $\gamma\in\A(G)$,
$$A_n(\gamma)\leq A_{n+1}(\gamma)\text{, and hence }A(\gamma)=\lim_{n\to\infty}A_n(\gamma).$$
It proves the property (i).

Let us verify (ii). Assume that $\gamma\in \A(G)$ is an arc. Since $G$ is a tame graph, $\gamma$ is not nowhere dense and then, because $Q$ is dense in $G$, there is $n\in\mathbb N$ and $[i_0\cdots i_n]\in\P_n$ such that $[i_0\cdots i_n]\subseteq \gamma$. Since by Remark~\ref{r:5} $v_{i_n}>0$, we obtain that \begin{equation*}\label{e:16}A_n(\gamma)>0\text{, and hence also } A(\gamma)>0.\end{equation*}

(iii) We want to verify that the arc-map $A$ is additive on $\gamma$. Clearly $A(\{x\})=0$. Suppose now that $\gamma\in\A(G)$ is nondegenerate  and $t\in \gamma$ is given. Let us distinguish two cases for $t$.

If $t\in Q$ then there is $n\in\mathbb N$ such that $t\in P_n$. Now for $m\geq n$
$$A_m(\gamma)=A_m(\gamma_t)+A_m(\gamma_{1-t}).$$
Hence in the limits for $m\to\infty$ we get the required equality
$$A(\gamma)= A(\gamma_t)+A(\gamma_{1-t}).$$

If $t\notin Q$, then there is for every $n\in\mathbb N$ exactly one arc $[i_0\cdots i_n]\in\P_n$, containing $t$ in the interior. It follows that

$$A_n(\gamma)=A_n(\gamma_t)+A_n(\gamma_{1-t})+\Delta([i_0\cdots i_n])$$
and clearly it is sufficient to show that
\begin{equation}\label{e:14}
\lim_{n\to\infty}\Delta([i_0\cdots i_n])=0.
\end{equation}
Let us call the sequence $i_0 i_1 \cdots$ the \emph{itinerary} of $t$. From (\ref{welldefinedDelta}) it follows that
$$\Delta([i_0\cdots i_n])\ge \Delta([i_0\cdots i_{n+1}])\text{ for each }n.$$

Suppose first that some symbol $j$ occurs infinitely often in the itinerary of $t$.  If $f(j)$ is a single partition arc, then it also occurs infinitely often in the itinerary of $t$, and we may replace $j$ by $f(j)$.  Since the mixing hypothesis does not allow for a cycle of partition arcs, we may conclude (after making finitely many such replacements) that $f(j)$ contains more than one partition arc. Among all the partition arcs contained in $f(j)$ there must exist some $k$ with $\frac{v_k}{\lambda_j v_j}$ maximal; call this ratio $c$ and notice that $c<1$.  We have
\begin{equation*}
\Delta([i_0 \cdots j k]) = \frac{v_k}{\lambda_j v_j} \Delta([i_0 \cdots j])= c\cdot \Delta([i_0 \cdots j]).
\end{equation*}
Therefore our decreasing sequence shrinks by the factor $c$ or better infinitely many times, and hence converges to zero.

Suppose now that each symbol in the itinerary of $t$ occurs only finitely often.  By hypothesis, we have $\lambda_j=\lambda>1$ for all but finitely many arcs $j\in\P$.  Since these arcs can occur only finitely many times in the itinerary, we find that the denominator in the expression for $\Delta([i_0 \cdots i_n])$ in \eqref{e:8} is eventually monotone increasing with respect to $n$, growing by a factor of $\lambda$ at each step.  Moreover, the numerator in this expression is always bounded by $1$.  Therefore the limit (\ref{e:14}) is zero, as desired.

Let us prove the property (iv). Since by (iii) the arc-map $A$ is additive, we can assume that $x$ is an endpoint of every $\gamma^m$. 
Let $\P_n(\gamma^m)$ denote the set of all elements of $\P_n$ nonintersecting $(\gamma^m)^\circ$.
There are two cases. Either there is $n$ such that $\P_n = \bigcup_m \P_n(\gamma^m)$, or for every $n$ there is an arc $\alpha^n \in \P_n \setminus \bigcup_m \P_n(\gamma^m)$.
In the first case, we can miss any finitely many members of $\P_n$ by an arc $\gamma^m$, and so for every $\eps > 0$ there is a finite family $\F \subset \P_n$ such that 
\[
	\sum_{[i_0\cdots i_n] \in \F} \Delta([i_0\cdots i_n]) \geq 1 - \eps
\]
by (\ref{e:5}) and an arc $\gamma^m$ such that $\F \subset \P_n(\gamma^m)$, and so $A_{n'}(\gamma^m) \leq \eps$ for every $n' \geq n$ by (\ref{welldefinedDelta}).
In the second case, $\alpha^n \cap (\gamma^m)^\circ \neq \emptyset$ for every $m$, and hence there is $m_n$ such that $\gamma^m \subset \alpha^n$ for every $m \geq m_n$. It follows that there is an itinerary $i_0 i_1 \cdots$ of $x$ such that $\alpha^n = [i_0\cdots i_n]$. Therefore, $A(\gamma^{m_n}) \leq \Delta([i_0 \cdots i_n])$, and we use (\ref{e:14}).
In both cases we have $\lim_m A(\gamma^m) = 0$.

The property (v) is clear when there are finitely many arcs in $A_{xy}$. Assume to the contrary that for some distinct points $x,y\in G$ and a sequence $(\alpha^m)_{m}\subset \A_{xy}$,
$$\forall~m\colon~A(\alpha^m)>A(\alpha^{m+1})\text{ and }\lim_{m\to\infty}A(\alpha^m)= \inf\{A(\beta)\colon~\beta\in \A_{xy}\}.
$$
It is known that the hyperspace $(C(G),H_d)$, where the Hausdorff metric $H_d$ is induced by the Euclidean metric $d$ in $\mathbb R^3$ restricted to $G$, is compact \cite[Theorem 4.17]{Nad92}. So we can assume that
\begin{equation}\label{e:20}\lim_mH_d(\alpha^m,C)=0\end{equation} for some subcontinuum $C$ of $G$. Clearly $x,y\in C$, and $C$ contains some $\alpha\in \A_{xy}$ since $G$ is hereditarily locally connected. In what follows we show that the arc $\alpha$ satisfies (\ref{e:24}).

From Proposition \ref{p:1}(ii) we have $A(\alpha)>0$. Fix an $\eps>0$. For an arc $\delta$ let $\delta(uv)$ denote the unique subarc of $\delta$ with endpoints $u,v\in\delta$ and for $n>0$  put
$$S(\delta,u,v,n)=\{[i_0\cdots i_n]\in\P_n\colon~[i_0\cdots i_n]\subseteq \delta(uv),~[i_0\cdots i_n]\cap P_0=\emptyset\}.$$
By the definition of the arc-map $A$ there is $n\in\bbn$ such that $A_n(\alpha)>A(\alpha)-\eps$, so using (\ref{e:6}) and a larger $n$ (if necessary) we can consider points $u,v\in\alpha^{\circ}$ such that for some finite subset $S'\subset S(\alpha,u,v,n)$
\begin{equation}\label{e:23}
\sum_{[i_0\cdots i_n]\in S'}\Delta([i_0\cdots i_n])>A(\alpha)-\eps.
\end{equation}
Since the set $\bigcup S'$ is a compact subset of $\alpha^{\circ}$ nonintersecting $P_0$, for each $x\in\bigcup S'$ there exists a open neighbourhood $U(x)$ for which
$$
\overline{U(x)}\cap P_0=\emptyset,~U(x)\cap\alpha=U(x)\cap G\subset\alpha^{\circ};
$$
then it follows from (\ref{e:20}) that in fact for each $U(x)$ there exists $m_x$ such that for each $m>m_x$
$$
U(x)\cap\alpha=U(x)\cap\alpha^m=U(x)\cap G;
$$
Since $\{U(x)\colon~x\in \bigcup S'\}$ is an open cover of the compact set $\bigcup S'$ we can consider its finite subcover $\{U(x_1),\dots,U(x_k)\}$ and put $m_1=\max_{1\le i\le k}m_{x_i}$. By the previous, for each $m>m_1$ we obtain
$\alpha^m\supset\bigcup S'$, and hence by (\ref{e:22}) and (\ref{e:23})
$$A(\alpha^m)\ge A_n(\alpha^m)>A(\alpha)-\eps.$$
Since the epsilon was arbitrary, it proves the conclusion of the claim.
\end{proof}

Now that we have our arc map $A$ and its basic properties, we show that it behaves well with respect to our map $f$.


\begin{theorem}\label{t:1}Let $f\in\CPM(G)$ with a partition $\P$. Suppose there is a summable $\lambda$-subeigenvector $v=(v_i)_{i\in\P}$ which is deficient in only finitely many coordinates from the set $\P'\subset \P$. Consider the arc-map $A$ defined in (\ref{e:22}). Then:
\begin{itemize}
\item[(i)] If an arc $\alpha$ is contained in a single partition arc, $\alpha \subset i_0\in\P$, then $A(f(\alpha))=\lambda_{i_0}A(\alpha)$.
\item[(ii)] If an arc $\alpha$ has empty intersection with $i^{\circ}$ for each $i\in\P'$ and if $f\vert\alpha$ is monotone, then $A(f(\alpha))=\lambda A(\alpha)$.
\end{itemize}
There is also a metric $\varrho \colon G\times G\to [0,1]$ compatible with the topology $\tau$ such that:
\begin{itemize}
\item[(iii)] For each $x,y\in G$, $\varrho(f(x),f(y))\le\lambda\varrho(x,y)$.
\end{itemize}
\end{theorem}

\begin{proof}
To verify the property (i), suppose first that the endpoints of $\alpha$ are two points $x,y\in Q$.  Take $n$ minimal so that $x,y\in P_n$.  Notice that $f$ induces a bijective correspondence between the set of arcs $[i_0\cdots i_n]\in\P_n$ contained in the arc $\alpha$ and the set of arcs $[i_1\cdots i_n]\in\P_{n-1}$ contained in the arc $f(\alpha)$.  Using \eqref{expansionDelta} and taking sums, we find
\begin{align}\label{expansion}
	A(f(\alpha)) = \sum_{\substack{[i_1\cdots i_n]\subseteq f(\alpha)}} \Delta([i_1\cdots i_n])
		= \sum_{\substack{[i_0\cdots i_n]\subseteq \alpha}} \lambda_{i_0} \Delta([i_0\cdots i_n])
		= \lambda_{i_0} A(\alpha)
\end{align}
Since $Q$ is dense, the conclusion is true by Proposition \ref{p:1}(iii),(iv).

The property (ii) is a direct consequence of (i), the fact that $P_0$ is countable and $f$ is continuous.

Let us prove (iii). Using Proposition \ref{p:1}(v) let us define the map $\varrho\maps G\times G\to [0,1]$ by $\varrho(x,x)=0$, $x\in G$ and for $x\neq y$,
\begin{equation}\label{e:25}\varrho(x,y)=\min\{A(\alpha)\colon~\alpha\in \A_{xy}\}.
\end{equation}
So $\varrho(x,x)=0$, $\varrho(x,y)>0$ and $\varrho(x,y)=\varrho(y,x)$ for each distinct $x,y\in G$.

As before, for an arc $\delta$ we let $\delta(uv)$ denote the unique subarc of $\delta$ with endpoints $u,v\in\delta$. In order to show that $\varrho$ satisfies the triangle inequality, let $x, y, z\in G$ be arbitrary. There are arcs $\alpha,\beta,\gamma\in \A(G)$ with endpoints $x,y$; $y,z$; $x,z$ respectively and such that
 $$\varrho(x,y)=A(\alpha),~\varrho(y,z)=A(\beta),~\varrho(x,z)=A(\gamma).
 $$
 If $z\in \alpha$ then since $\alpha(xz)\subset \alpha$ we can write
  \begin{align*}\varrho(x,z)\le A(\alpha(xz))\le A(\alpha)=\varrho(x,y)\le \varrho(x,y)+\varrho(y,z).\end{align*}
  Assume that $z\notin\alpha$. There exists the unique point $t\in\alpha$ such that $\alpha\cap \beta(tz)=\{t\}$. Since $\alpha(xt)\cup\beta(tz)$ is an arc from $\A_{xz}$, using the additivity of $A$ we can write
  \begin{align*}
  &\varrho(x,z)=A(\gamma)\le A(\alpha(xt)\cup\beta(tz))=A(\alpha(xt))+A(\beta(zt))\le\\
  \le &A(\alpha)+A(\beta)=\varrho(x,y)+\varrho(y,z).
  \end{align*}

Thus $\varrho$ is a metric. Let us show that
\begin{equation}\label{e:15}
\lim_{n\to\infty}\max\{A(\alpha)\colon~\alpha\in\P_n\}=0.\end{equation}

We know that for each $\alpha\in\P_n$, $\alpha=[i_0\cdots i_n]$ and by (\ref{welldefinedDelta}) $A(\alpha)=A_n(\alpha)=\Delta([i_0\cdots i_n])$. So by (\ref{e:5}) every maximum in (\ref{e:15}) is less than $1$. Moreover, every element of $\P_{n+1}$ is contained in some element of $\P_n$, so
$$\max\{A(\alpha)\colon~\alpha\in\P_n\}\ge\max\{A(\alpha)\colon~\alpha\in\P_{n+1}\},$$ hence the limit in (\ref{e:15}) exists.
Assume to the contrary that it equals to some $\eps>0$.
Let $\T$ be the tree $\{\alpha \in \bigcup_n \P_n\colon~A(\alpha) \geq \eps\}$ ordered by inclusion.
By our assumption $\T$ is infinite, but by (\ref{e:5}) every level $\T \cap \P_n$ is finite.
Hence, $\T$ has an infinite branch $[i_0] \supset [i_0 i_1] \supset \cdots$, so $A([i_0\cdots i_n])=\Delta([i_0\cdots i_n])\ge\eps$ for each $n$. This contradicts (\ref{e:14}), and hence the equality (\ref{e:15}) is proved.

We need to verify that the metric $\varrho$ is compatible with the original topology $\tau$ on $G$.
This means to verify that the identity $(G,\tau) \to (G,\varrho)$ is a homeomorphism. Since $(G,\tau)$ is compact and $(G,\varrho)$ is clearly Hausdorff, it is enough to show that the identity is continuous. Thus suppose that we have a sequence $x_n$ converging to $x$ in the topology of $G$. We need to show that it converges also in the metric $\varrho$. Fix $\varepsilon>0$. We want to find $n_0$ such that $\varrho(x,x_n)<\varepsilon$ for $n>n_0$. By (\ref{e:15}) there is a $n_1\in\mathbb N$ such that $A(\alpha)< \varepsilon$ for each $\alpha\in\P_{n_1}$.
Let us distinguish several cases.

Suppose first that $x\notin P_{n_1}$. Consider the only element $\alpha\in\mathcal \P_{n_1}$ which contains $x$ in the interior (with respect to the topology $\tau$).
 Since $\alpha$ is a neighbourhood of $x$ there is $n_0$ such that $x_n\in \alpha$ for $n>n_0$. We have shown in the previous that the arc-map $A$ is additive. It implies that
 $$\varrho(x,x_n)\leq A(\alpha)<\varepsilon$$
  for $n> n_0$.

Second, suppose that $x\in P_{n_1}$ and let us distinguish two cases.
Suppose first that the point $x$ is isolated in $P_{n_1}$. Since the set $P_0$ as well as $P_{n_1}$ contains all the ramification points it follows from Remark \ref{r:6} that $x$ is a point of finite order $k$. Thus there are pairwise distinct $\alpha_1,\dots, \alpha_k\in \mathcal \P_{n_1}$ containing $x$. Since the set $\alpha_1\cup\dots\cup \alpha_k$ forms a neighbourhood of $x$ it follows that $x_n\in \alpha_1\cup\dots\cup \alpha_k$ for $n$ greater than some $n_0$. By the convexity of the arc-map $A$ it follows that
$$\varrho(x,x_n)\leq \max\{A(\alpha_1),\dots,A(\alpha_k)\}<\varepsilon$$
 for $n>n_0$.

In the remaining case $x$ is a limit point of $P_{n_1}$.
Since
$$\sum_{\alpha\in\P_{n_1}} A(\alpha)=\sum_{[i_0\cdots i_{n_1}]\in\P_{n_1}}\Delta([i_0\cdots i_{n_1}])=1,$$
there is a finite set $\mathcal F\subseteq \P_{n_1}$ such that
$$\sum_{\alpha\in\P_{n_1}\setminus \mathcal F}A(\alpha)<\varepsilon.$$
Let $\mathcal G=\{\alpha\in\mathcal F: x\in \alpha\}$.
The set
$$\bigcup \mathcal G\cup \bigcup (\P_{n_1}\setminus \mathcal F)\supseteq G\setminus \bigcup(\mathcal F\setminus \mathcal G)$$
 forms a neighbourhood of $x$. Thus there is $n_0$ such that $x_n$ is in this neighbourhood for $n > n_0$. It follows that for such $n$
$$\varrho(x,x_n)\leq \max\{A(\alpha)\colon~\alpha\in\mathcal G\}<\varepsilon$$
if $x_n\in\bigcup \mathcal G$ or
$$\varrho(x,x_n)\leq \sum_{\alpha\in \P_{n_1}\setminus\mathcal F} A(\alpha)<\varepsilon$$
 otherwise.
In any case $\varrho(x,x_n)<\varepsilon$ for $n>n_0$.
Thus the metric $\varrho$ is compatible with the original topology $\tau$ on $G$.

Let $x,y\in Q\cap G$. Let $\alpha\in\A_{xy}$ satisfy $\varrho(x,y)=A(\alpha)$ and $\beta\in\A_{f(x)f(y)}$.
Since each vertex of $G$ belongs to $P_0$ and $f(P_{n})\subset P_{n-1}\subset P_{n}$, every element of the set $\P(\beta,n-1)$ of arcs $[i_1\cdots i_n]\in\P_{n-1}$ contained in the arc $\beta$ has its preimages in the set $\P(\alpha,n)$ of arcs $[i_0\cdots i_n]\in\P_n$ contained in the arc $\alpha$; using (\ref{expansion}) and the definition of $\varrho$ we can write for each $n$
\begin{align}\label{expansion1} &\varrho(f(x),f(y))\le A(\beta)=\sum_{[i_1\cdots i_n]\in\P(\beta,n)}\Delta([i_1\cdots i_n])\le \\ \le &\sum_{[i_0\cdots i_n]\in\P(\alpha,n-1)} \lambda\Delta([i_0\cdots i_n]) \le \lambda A(\alpha)=\lambda\varrho(x,y).
\end{align}
Since $Q$ is dense in $G$ and $f$ is continuous, the inequality $\varrho(f(x),f(y))\le\lambda\varrho(x,y)$ holds true for every pair $x,y\in G$.
\end{proof}

\subsection{Countably affine graphs}

Now we construct the countably affine graph $G'$. As before we have $f\in\CPM(G)$ with partition $\P$, a summable $\lambda$-subeigenvector $v$ with finite deficiency, and the arc map $A$ defined in~\eqref{e:22}. Then we have:

\begin{theorem}\label{t:graph}
There exists a countably affine graph $G'\subset \mathbb{R}^3$ and a homeomorphism $\phi\maps G\to G'$ such that $\ell(\phi(\alpha))=A(\alpha)$ for all arcs $\alpha\in\A(G)$.
\end{theorem}

\begin{proof}
Consider the set $R=\prod_{k\ge 1}[0,1/2^{k+1}]$ equipped with the supremum metric $\tilde d$. Clearly, the metric $\tilde d$ is compatible with the  Tychonoff product topology on $R$ and $(R,\tilde d)$ is a compact metric space. Let $(X,m)$ be a countable metric space with the metric $m$, to avoid trivial cases and to simplify the notation assume that $X=\{x_n\colon~n\ge 1\}$ is infinite and $\dia_m(X)\le 1$.

Using the metric $m$ and an element $r=(r_k)_{k\ge 1}\in R$ we can define the map
$$\pi_r\maps X\to\bbr,~ \pi_r(x_n)=\sum_{k\ge 1}r_k~m(x_n,x_k).$$

The following facts are easy to verify and we leave their proof to the reader.

\begin{itemize}
\item[($+$)] \label{l:1} For each $r\in R$ and $m,n\in\bbn$, $$\vert\pi_r(x_m)-\pi_r(x_n)\vert\le\frac{m(x_m,x_n)}{2}.$$
In particular, the map $\pi_r\maps (X,m)\to (\bbr,\vert\cdot\vert)$ is continuous.
\item[($++$)] \label{p:3} Let $G_n\subset R$, $n\in\bbn$, be defined as
$$G_n=\{r=(r_k)_{k\ge 1}\in R\colon~\pi_{r}(x_i)\neq \pi_{r}(x_j)\text{ for each }1\le i<j\le n\}.$$ (i) For each $n$, the set $G_n$ is open and dense in $(R,\tilde d)$.
~(ii) For every $r\in G=\bigcap_{n\ge 1}G_n$, the map $\pi_r$ is injective.
\end{itemize}

We assume that $f\in\CPM(G)$ with a partition $\P$. In particular, the $\tau$-closed countable set $P=P_0$ -- see before Remark~\ref{r:6} -- contains all vertices and endpoints of $G$ and intersects any simple closed curve in G in at least two points.  Due to Remark \ref{r:4} we can assume that $P$ is countably infinite. We want to apply ($+$) and ($++$); to this end put $$X=P\text{ and }m=\varrho,$$
where the metric $\varrho$ on $G$ was guaranteed by Theorem \ref{t:1}(iii).  We know that the metric $\varrho$ is compatible with the original topology $\tau$ on $G$, so the space $(P,\varrho)$ is a countable compact metric space. Using ($++$), fix an $r\in R$ for which the map $\pi_r$ is injective. Since by ($+$) the map $\pi_r$ is continuous, it is in fact a homeomorphism of $P$ onto its image $\pi_r(P)$. In order to construct a needed graph $G'$, we will use a sheaf of planes $\Sigma=\{\sigma_i\colon~i\in\P\}$ in $\bbr^3$ containing the real line $\bbr$. To make our construction as simple as possible, we will assume that $\Sigma$ is convergent, i.e., $\lim_{i\in\P}n_i=n$, where each $n_i$ is a normal vector of $\sigma_i$ and $n$ is a nonzero vector in $\bbr^3$.

By the definition of the partition $\P$ each $i\in\P$ is an element of some $\A_{x_iy_i}(G)$, $x_i,y_i\in P$. Then (\ref{e:25}) and ($+$) imply
$$A(i)\ge \varrho(x_i,y_i)>\frac{\varrho(x_i,y_i)}{2}\ge \vert \pi_r(x_i)-\pi_r(y_i)\vert,$$
so we can join the points $\pi_r(x_i),\pi_r(y_i)\in\bbr$ by a zig-zag $j_i$ (a piecewise affine curve consisting of finitely many segments) of length $A(i)$ and placed to the plane $\sigma_i$. We know from (\ref{e:5}) that $\lim_{i\in\P}A(i)=0$, hence by the previous also
$$
\lim_{i\in\P}\vert\pi_r(x_i)-\pi_r(y_i)\vert=0.
$$
It means that
$$\overline{\bigcup_{i\in\P}j_i}\setminus \bigcup_{i\in\P}j_i=\pi_r(P\setminus\bigcup_{i\in \P}i)\subset \pi_r(P);$$ it follows that the set
$$
G'=\overline{\bigcup_{i\in\P}j_i},
$$
is a countably affine graph in $\bbr^3$. Let us define the map
$\phi\maps G\to G'$ by
\begin{equation*}\phi(z)=\begin{cases}
\pi_r(z)~\text{ if }z\in P,\\
u\in j_i~\text{ satisfying }\ell(j_i(\pi_r(x_i)u))=A(i(x_iz))\text{ if }i\in\P\text{ and }z\in i.
\end{cases}
\end{equation*}
By the above definition, $\phi$ is a homeomorphism and for each arc $\alpha\in\A(G)$,  $\phi(\alpha)\in\A(G')$ and $A(\alpha)=\ell(\phi(\alpha))$.
\end{proof}

\subsection{Conclusion of the proof}

Having used our arc map to construct a suitable countably affine graph, we are ready to finish the construction of the conjugate map of constant (bounded) slope.

\begin{proof}[Proof of Theorem~\ref{t:main}~(ii) and Theorem~\ref{t:broad}~(ii).]
We continue to use the arc map $A$ and the homeomorphism $\phi\maps G\to G'$ constructed above.
Put $g=\phi\circ f\circ\phi^{-1}$. It follows immediately from Theorem~\ref{t:1} and Theorem~\ref{t:graph} that $g$ is piecewise affine with slope $\lambda_i$ on each piece of the zig-zag $\phi(i)\subset G'$, $i\in\P$. Thus $g$ is the desired conjugate map of bounded slope $\lambda$, and if $v$ is a $\lambda$-eigenvector (no deficiency), then $g$ has constant slope $\lambda$.
\end{proof}

In fact, we have proved something slightly stronger, which turns out to be useful in Section~\ref{s:lipschitz}.
Consider the map $L\maps G'\times G'\to\bbr$ defined by
$$L(x,y)=\min\{\ell(\alpha)\colon~\alpha\in\A_{xy}(G')\}.$$
From Theorems~\ref{t:1} and~\ref{t:graph} we have $L(\phi(x),\phi(y))=\rho(x,y)$ for all $x,y\in G$. This implies:

\begin{corollary}\label{c:isometry}
The function $L$ is a metric on $G'$ compatible with its topology and the homeomorphism $\phi\maps (G,\rho)\to(G',L)$ is an isometry.
\end{corollary}

\subsection{Discussion and an open question}

The main question left open by our work is the following:

\begin{question}
Give a condition on the transition matrix which is both necessary and sufficient for a map $f\in\CPM(G)$ to be conjugate to a countably affine graph map of constant slope.
\end{question}

It is easy to see that condition (ii) in Theorem~\ref{t:main} is not necessary. It is enough to construct a countably affine graph $G$ of infinite total length and a constant slope map $f\in\CPM(G)$, whose transition matrix $M$ has no summable eigenvectors. This was done in~\cite{MiRo17} on the extended real line $[-\infty,\infty]$.

Finally, we construct an example to show that condition (i) in Theorem~\ref{t:main} is not sufficient. It is a modification of a piecewise-continuous interval map from~\cite{MiRo17}, made continuous by working on a dendrite.

\begin{example}\label{example1}
Consider $\mathbb{R}^2$ with polar coordinates, writing $(r,\theta)$ in place of $(r\cos\theta,r\sin\theta)$. Fix two decreasing sequences $r_0>r_1>\cdots$ and $\theta_0>\theta_1>\cdots$ with $\theta_0<\pi$ and with limits $\lim_n r_n=0$, $\lim_n\theta_n=0$. Consider the tame graph $G\subset\mathbb{R}^2$ formed as the union of the arcs $A_n=\left\{(r,\theta_n);~0\leq r\leq r_n\right\}$, $n\geq0$; $C_n=\left\{(r,-\theta_n);~0\leq r\leq r_n\right\}$, $n\geq1$; and $B=\left\{(r,-\theta_0);~0\leq r\leq r_0\right\}$.
$G$ is a dendrite with a single branchpoint. Moreover, $G$ is a locally connected fan, and we will refer to the arcs $A_n$, $B$, $C_n$ as the ``blades'' of the fan. We define the map $f\maps G\to G$ so that the branchpoint is fixed, $C_1$ maps affinely onto $B$, $B$ maps affinely onto $A_0$, and on each of the remaining blades $f$ is the piecewise affine ``connect-the-dots'' map with the following dots:
\begin{equation*}
\arraycolsep=1.4pt
\begin{array}{rclcrcl}
\multicolumn{3}{c}{C_{n+1}\xrightarrow{2:1}C_n}&\hspace{3em}&
\multicolumn{3}{c}{C_{2^n}\xleftarrow{2:1} A_n \longrightarrow A_{n+1}} \\ \midrule
f(0,0)&=&(0,0) &&
f(0,0)&=&(0,0) \\
f(r_{n+1}/2,-\theta_{n+1})&=&(r_n,-\theta_n) &&
f(r_n/3,\theta_n)&=&(r_{2^n},-\theta_{2^n}) \\
f(r_{n+1},-\theta_{n+1})&=&(0,0) &&
f(2r_n/3,\theta_n)&=&(0,0) \\
&&&&
f(r_n,\theta_n)&=&(r_{n+1},\theta_{n+1}) \\
\end{array}
\end{equation*}
Figure~\ref{fig:dendritemap} illustrates the map $f$. It is straightforward to verify that if $U$ is an open set, then there is $n_0\geq0$ such that $f^{n_0}(U)$ contains a whole blade, and then $\bigcup_{n=0}^\infty f^n(U)=G$ is the whole dendrite. This shows that $f$ is topologically mixing.

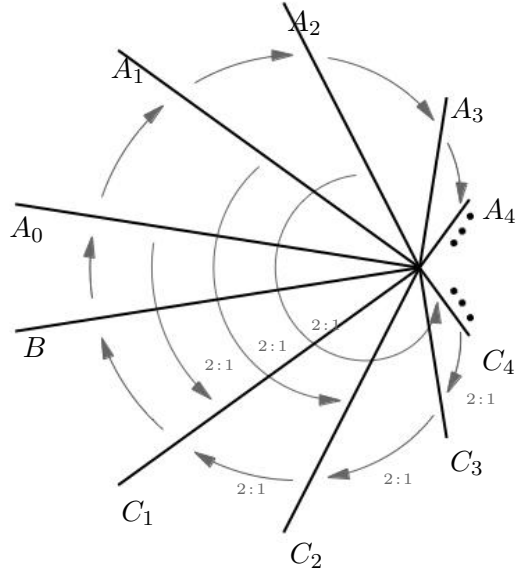
\begin{figure}[htb!!]
\centering
\includegraphics[width=.5\textwidth]{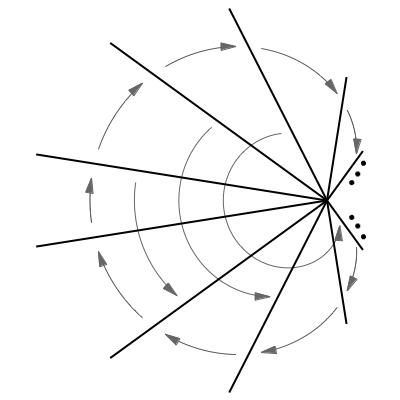}
\begin{picture}(0,0)
\put(-18,65){$C_4$}
\put(-30,27){$C_3$}
\put(-90,-8){$C_2$}
\put(-153,8){$C_1$}
\put(-190,70){$B$}
\put(-195,115){$A_0$}
\put(-157,175){$A_1$}
\put(-91,193){$A_2$}
\put(-30,160){$A_3$}
\put(-18,122){$A_4$}
\put(-24,52){\tiny\color{gray}$2\!:\!1$}
\put(-55,22){\tiny\color{gray}$2\!:\!1$}
\put(-110,18){\tiny\color{gray}$2\!:\!1$}
\put(-122,65){\tiny\color{gray}$2\!:\!1$}
\put(-102,72){\tiny\color{gray}$2\!:\!1$}
\put(-82,80){\tiny\color{gray}$2\!:\!1$}
\end{picture}
\caption{A dendrite map with no conjugate map of constant slope.}\label{fig:dendritemap}
\end{figure}

We will use the partition into blades as a \emph{slack} Markov partition~\cite{BoBru15}. The corresponding transition matrix~\cite{BoBru15} admits (up to scaling) exactly one non-negative eigenvector with eigenvalue $\lambda=2$. Its entries are $v_{A_n}=2^n+1$, $v_B=1$, and $v_{C_n}=\frac12$. The choice to work with the eigenvalue $\lambda=2$ is intentional, as it can be shown (cf.~\cite{BoBru15}) that $h(f)=\log 2$.

Now suppose there is a homeomorphism $\phi\maps G\to G'$ onto a countably affine graph $G'$ which conjugates $f$ with a map $g$ of constant slope $2$. The constant slope condition implies that up to rescaling, the lengths of the blades of $G'$ must be given by the entries of the eigenvector $v$. The construction of this dendrite $G'$ is no problem, but as soon as we put a constant slope map $g$ there with the right Markov dynamics, we find an obstruction to the conjugacy: the set $\bigcap_{n=0}^{\infty} g^{-n}(\phi(A_n))$ is an arc of length
\begin{equation*}
2\times\frac{3}{4}\times\frac{5}{6}\times\cdots\times\frac{2^n+1}{2^n+2}\times\cdots=1,
\end{equation*}
(cf.~\cite{MiRo17}) which means that the map $g$ is not topologically mixing, and therefore cannot be topologically conjugate to $f$.
\end{example}


\section{Entropy and horseshoes for maps from $\CPM(G)$}\label{s:horseshoes}

For a matrix $M=(m_{ij})_{i,j\in\P}$ we can consider the powers $M^n=(m_{ij}(n))_{i,j\in\P}$ of $M$:
\begin{equation}\label{e:11}M^0=I=(\delta_{ij})_{i,j\in \P},~M^n=\left (\sum_{k\in \P}m_{ik}m_{kj}(n-1)\right )_{i,j\in \P}, ~n\in\mathbb N.\end{equation}

\begin{proposition}\label{p:11}
Let $f\in\CPM(G)$ with $M(f) = (m_{ij})_{i,j\in \P}$.
\begin{itemize}
\item[(i)] For each $n\in\mathbb N$ and $i,j\in\P$, the entry $m_{ij}(n)$
of $M^n$ is finite.
\item[(ii)] The entry $m_{ij}(n) = m$ if and
only if there are exactly $m$ arcs $[ii^k_1\cdots i^k_{n-1}j]\in\P_n$, $k\in\{1,\cdots,m\}$, for which $f^n([ii^k_1\cdots i^k_{n-1}j])=j$, $k=1,\dots,m$.
\end{itemize}
\end{proposition}

\begin{proof}(i) From the continuity of $f$ and the definition of $M(f)$ follows that the sum
$\sum_{i\in\P}m_{ij}$ is finite for each $j\in\P$, which directly
implies (i). ~(ii) For $n=1$ this is given by the relation (\ref{e:3}) defining the matrix $M=M(f)$. The induction step follows immediately from the definition (\ref{e:11}) of the product of the nonnegative matrices $M$ and $M^{n-1}$ and Lemma \ref{lem:refine}(iii).
\end{proof}

\begin{corollary}\label{c:3} With the help of Proposition \ref{p:11} one can show that for each $f\in\CPM(G)$, the set $\Per(f)$ of periodic points is dense in $G$ and $\htop(f)>0$.
\end{corollary}


Let us recall that a matrix $M=(m_{ij})_{i,j\in\A}$, where the index set $\A$ is finite or countably infinite, is
\begin{itemize}\item \emph{irreducible},
if for each pair of indices $i,j$ there exists a positive integer $n$ such
that $m_{ij}(n)>0$,
\item \emph{aperiodic}, if for each index $i\in\A$ the value $p(i)=\gcd\{\ell\colon~m_{ii}(\ell)>0\}$ equals to \emph{one}.
\end{itemize}

\begin{remark}For $f\in\CPM(G)$ with Markov partition $\P$ its transition matrix $M=M(f)=(m_{ij})_{i,j\in\P}$ is irreducible and aperiodic.
\end{remark}

In the sequel we follow the approach suggested by Vere-Jones \cite{Ver-Jo62}.

\begin{proposition}\label{p:2}Let $M=(m_{ij})_{i,j\in\A}$ be
a nonnegative irreducible aperiodic matrix indexed by a countable index set $\A$. There exists a
common value $\lambda_M$ such that for each $i,j$
\begin{equation*}\label{e:13}\lim_{n\to\infty} [m_{ij}(n)]^{\frac{1}{n}}=\lambda_M.\end{equation*}
\end{proposition}

The \emph{Gurevich entropy} of a nonnegative irreducible aperiodic matrix $M$ is defined as
\begin{equation*}\label{e:12}
h(M) = \sup\{\log r(M') : M' \text{ is a finite transition submatrix of } M \},
\end{equation*}
where $r(M')$ is the largest eigenvalue of the finite transition matrix $M'$.

We can also ask about the entropy of the corresponding Markov shift $(\Sigma_M,\sigma)$. When $\P$ is infinite, this system is noncompact, so there are many possible notions of entropy. We will write $h(\Sigma_M)$ to denote the supremum of entropies of ergodic shift-invariant Borel probability measures.

In \cite{Gur69} Gurevich proved the following proposition.
An accessible proof in English can be found in \cite{Ki98}.

\begin{proposition}\label{p:4}
$h(\Sigma_M) = h(M) = \log \lambda_M$.
\end{proposition}

We wish to interpret Gurevich entropy in the context of a graph map $f\in\CPM(G)$. To do it, we need to understand more closely the connection between the dynamical system $(G,f)$ and its symbolic dynamics $(\Sigma_M,\sigma)$. From the point of view of topological dynamics, the relationship is given by the map
\begin{equation}\label{e:psi}
\psi\maps\Sigma_M\to G, \quad \psi(i_0i_1\cdots) = x \text{ where } \left\{x\right\} = \bigcap_{n=0}^\infty [i_0i_1\cdots i_n].
\end{equation}
Using Lemma~\ref{lem:refine} it is easy to verify that $\psi$ is well-defined, continuous, and that $\psi\circ\sigma=f\circ\psi$. Unfortunately, $\psi$ in general is neither injective nor surjective, so we cannot speak of a conjugacy or semiconjugacy.

Nevertheless, the relationship between our graph map and its symbolic dynamics is useful from the point of view of ergodic theory. The following definition is due to Hofbauer \cite{Hof79}.

\begin{definition}
In a measurable dynamical system $(X,T)$ a measurable set $N\subset X$ is called \emph{small} if it is backward invariant
 ($T^{-1}(N)\subset N$) and if every invariant probability measure $\mu$ concentrated on $N$ has entropy $h_\mu=0$. Two measurable dynamical systems $(X,T)$, $(X', T')$ are \emph{isomorphic modulo small sets} if there exist small sets $N\subset X$, $N'\subset X'$ and a bimeasurable bijection $\psi\maps X\setminus N \to X'\setminus N'$ such that $T'\circ\psi=\psi\circ T$.
\end{definition}

\begin{lemma}\label{l:null}
A small set is a measure zero set with respect to every ergodic invariant probability measure of positive entropy.
\end{lemma}
\begin{proof}
Because a small set is invariant, each ergodic measure assigns to it either full or zero measure. If it is full, then the entropy must be zero.
\end{proof}

Our dynamical systems $(G,f)$, $(\Sigma_M, \sigma)$ become measurable dynamical systems as soon as we equip them with their Borel $\sigma$-algebras. Then we have:

\begin{theorem}
\label{t:iso} Let $f\in\CPM(G)$ with a partition set $\P$. Then $(G,f)$ and $(\Sigma_M,\sigma)$ are isomorphic modulo small sets.
\end{theorem}
\begin{proof}
Recall the definition of $\psi$ in~(\ref{e:psi}) and $Q$ in Lemma~\ref{lem:refine}. We will show that
\begin{enumerate}[label=(\roman*)]
\item\label{it:small} $Q\subset G$ and $N=\psi^{-1}(Q)\subset\Sigma_M$ are small sets,
\item\label{it:bij} The restricted map $\psi\maps\Sigma_M\setminus N \to G\setminus Q$ is a bijection, and
\item\label{it:meas} With respect to the Borel $\sigma$-algebras, both $\psi\maps\Sigma_M\setminus N \to G\setminus Q$ and its inverse are measurable.
\end{enumerate}
We start with~\ref{it:small}. Backward invariance of $Q$ is clear from the formula $Q=\bigcup_{n=0}^\infty f^{-n}(P)$. $N$ inherits backward invariance from $Q$ because of the relation $f\circ\psi=\psi\circ\sigma$. By Lemma~\ref{lem:refine}~\ref{it:mar}, $Q$ is countable, and therefore every invariant measure concentrated on $Q$ has entropy zero. Finally, we argue that $N$ is also countable. Let $x\in P_n$. It is enough to show that any two points $i=i_0i_1\cdots$, $i'=i'_0i'_1\cdots$ in $\psi^{-1}(x)$ with $i_0\cdots i_n=i'_0\cdots i'_n$ must be equal. We do it by induction. Suppose $i_0\cdots i_t=i'_0\cdots i'_t$ with $t\geq n$. Since $x\in P_n\subset P_t$, we see that $x$ is an endpoint of the $\P_t$ partition arc $[i_0\cdots i_t]$. Both of the subarcs $[i_0\cdots i_{t+1}]$, $[i'_0\cdots i'_{t+1}]$ also contain $x$, and therefore their interiors are not disjoint. Since these are $\P_{t+1}$ partition arcs, we get $i_{t+1}=i'_{t+1}$. This concludes the induction step and the proof that $N$ is countable.

We prove~\ref{it:bij} by giving the formula for the inverse map. It is the so-called itinerary map $\phi\maps G\setminus Q \to \Sigma_M\setminus N$ given by setting $\phi(x)=i_0i_1\cdots$ if $f^n(x)\in i_n\in\P$ for all $n\geq0$. It is just a matter of checking the definitions to see that $\phi(x)=i_0i_1\cdots$ if and only if $x\in[i_0\cdots i_n]$ for all $n$, which happens if and only if $\psi(i_0i_1\cdots)=x$. Thus $\phi,\psi$ are inverses to each other.

Next we prove~\ref{it:meas}. We already know that $\psi\maps \Sigma_M\to G$ is continuous. Therefore the preimage of a relatively open subset $U\subset G\setminus Q$ is relatively open in $\Sigma_M\setminus N$. This shows that the restricted map $\psi\maps \Sigma_M\setminus N \to G\setminus Q$ is also continuous and therefore measurable. To prove measurability of the inverse, we will show that $\phi$ is continuous at each point where it is defined. Fix a point $x\in G\setminus Q$ and write $\phi(x)=i_0i_1\cdots$. Fix $n\geq0$. We will show that $x$ has a neighborhood $U$ in the graph $G$ such that $U\subset[i_0\cdots i_n]$. Since $x\notin P_n$ and $P_n$ is closed, we know that $x$ has a neighborhood $V$ in $G$ with $V\cap P_n=\emptyset$. By local connectedness of $G$ there is a connected neighborhood $U$ of $x$ contained in $V$. Since $U$ is connected and does not intersect $P_n$, it must be contained in a single arc of $\P_n$, and therefore $U\subset [i_0\cdots i_n]$, as desired.
\end{proof}

\begin{corollary}\label{c:entropy}
If $f\in\CPM(G)$ has transition matrix $M$, then $\htop(f)=\log\lambda_M$.
\end{corollary}
\begin{proof}
Consider the sets $\mathcal{E}_+(\Sigma_M,\sigma)$, $\mathcal{E}_+(G,f)$ of positive-entropy ergodic invariant Borel probability measures on our two systems. If we take the supremum of entropies of measures over these two sets we get $\log\lambda_M$ and $\htop(f)$, respectively; this uses Proposition~\ref{p:4}, the variational principle for continuous maps on the compact space $G$, and the fact that both $\Sigma_M$ and $f$ have positive entropy. By Lemma~\ref{l:null}, our isomorphism modulo small sets $\psi$ induces a bijection $\psi_*\maps \mathcal{E}_+(\Sigma_M,\sigma)\to\mathcal{E}_+(G,f)$, $\psi_*(\mu)=\mu\circ\psi^{-1}$ which preserves entropy: $h_{\psi_*\mu}(f)=h_\mu(\sigma)$.
\end{proof}


\begin{theorem}\label{t:6}Let $f\in\CPM(G)$.
Then there is a sequence $(s_n)_n$ of positive integers such that $f^n$ has an $s_n$-horseshoe and $$\lim_{n}\frac{1}{n}\log s_n=\htop(f).$$
\end{theorem}
\begin{proof}
Choose a partition $\P$ and let $M$ be the associated transition matrix. Fix a partition arc $j\in\P$.
We use Proposition~\ref{p:2}, Proposition~\ref{p:11}(ii), and Corollary~\ref{c:entropy}. By those
statements, $\htop(f)=\log (\lim_n [m_{jj}(n)]^{\frac{1}{n}})$ and for each $n$, the arc $j$
contains $m_{jj}(n)$ arcs $j_1,\dots,j_{m_{jj}(n)}$ with
pairwise disjoint interiors such that $f^n(j_i)\supset j_{\ell}$ for all $1\le i,\ell\le
m_{jj}(n)$. Clearly, the map $f^n$ has an $m_{jj}(n)$-horseshoe \cite{Mi79}.
Let $s_n=m_{jj}(n)$ and the proof is finished.
\end{proof}

\begin{corollary}\label{c:2}Similarly as for the interval case \cite{ALM00} one can deduce from Theorem \ref{t:6} the following.
\begin{itemize}
\item[(i)] $\liminf_{n}\frac{1}{n}\log \Var(f^n)\ge\htop(f)$, where $\Var(f)$ denotes the variation of $f$ on $G$,
\item[(ii)] $\htop(f)=\sup\{\htop(f\upharpoonright M)\colon~M\subset G\text{ is minimal}\}$,
\item[(iii)] $\limsup_{n}\frac{1}{n}\log \card\{x\in G\colon~f^n(x)=x\}\ge\htop(f)$.
\end{itemize}
\end{corollary}

%

\section{Entropy, Hausdorff dimension and Lipschitz constants}\label{s:lipschitz}

Let $(X,d)$ be a nonempty compact metric space with Hausdorff dimension $\HD_d(X)$, and let $f\maps X\to X$ be a Lipschitz continuous map with Lipschitz constant $\Lip_d(f)=\sup_{x\neq y} d(f(x),f(y)) / d(x,y)$. We write $\log^+(x)$ to denote the maximum of $\log(x)$ and $0$. The following inequality is well known \cite{DaZhGe98},\cite{Mi03}.

\begin{proposition}
An upper bound for the topological entropy is given by
$$\htop(f) \leq \HD_d(X) \cdot \log^+\Lip_d(f).$$
\end{proposition}

Replacing the metric on $X$ with another compatible metric can change both the Hausdorff dimension and the Lipschitz constant. Thus, a natural question arises: by varying the metric $d$, can we make the product $\HD_d(X) \cdot \log^+ \Lip_d(f)$ as close to $\htop(f)$ as we like? Our construction of conjugate maps of bounded slope gives us a way to address this question for tame graph maps.

Since by Proposition~\ref{p:2} the value $\lambda_M^{-1}$ is a common radius of convergence of the power series $M_{ij}(z)=\sum_{n\ge 0}m_{ij}(n)z^n$, we immediately obtain for each pair $i,j\in\A$,
\begin{equation*}\label{e:10}M_{ij}(\lambda^{-1})\begin{cases}
\in\bbr,~\lambda>\lambda_M,\\
=\infty,~ \lambda<\lambda_M.
\end{cases}
\end{equation*}

The following result first used in \cite{Pr64} will be useful when proving our theorem.

\begin{proposition}Let $f\in\CPM(G)$ with a partition $\P$, consider its transition matrix $M=(m_{ij})_{i,j\in\P}$. For each $\lambda>\lambda_M$ and $j\in\P$,
\begin{equation}\label{e:17}
\forall~i\in\P\colon~\sum_{k\in\P}m_{ik}M_{kj}(\lambda^{-1})=\lambda M_{ij}(\lambda^{-1})\left (1-\frac{\delta_{ij}}{M_{ij}(\lambda^{-1})}\right )\le \lambda M_{ij}(\lambda^{-1}).
\end{equation}
\begin{proof}See \cite[Theorem 1]{Pr64}. The right inequality follows from the fact that for irreducible $M$, $M_{jj}(\lambda^{-1})>1$ for each $j\in\P$.\end{proof}
\end{proposition}

\begin{proposition}\label{p:sub}Let $f\in\CPM(G)$ with a partition $\P$, consider its transition matrix $M=(m_{ij})_{i,j\in\P}$. For $\lambda>\lambda_M$ and $j\in\P$ put $v=(v_{i}=M_{ij}(\lambda^{-1}))_{i\in\P}$. Then $v$ is a $\lambda$-subeigenvector which is deficient in the coordinate $j$ only. 
\end{proposition}
\begin{proof}Since the matrix $M$ is irreducible and Proposition \ref{p:11} is true, our definition of $\lambda_M$ gives $1\le \lambda_M<\lambda$. The Kronecker symbol $\delta_{ij}=0$ for $i\neq j$ and $\delta_{jj}=1$, so the inequality in (\ref{e:17}) is in fact an equality except when $i=j$.
\end{proof}

If a map $f\in\CPM(G)$ is leo, then by Remark~\ref{r:5}(ii) every $\lambda$-subeigenvector is summable. So by Proposition \ref{p:sub} for such a map, for each $\lambda>\lambda_M$ we have a summable $\lambda$-subeigenvector which is deficient in exactly one coordinate and Theorem \ref{t:1} applies.

\begin{theorem}\label{t:4} Let $f\in\CPM(G)$ be leo. Then for each $\epsilon>0$ there is a distance function $\varrho$ on $G$ compatible with the topology $\tau$ such that $$ \HD_\varrho(G) \cdot \log^+ \Lip_\varrho(f) < \htop(f) + \epsilon.$$
\end{theorem}

\begin{proof}
Fix a partition $\P$ and let $M$ be the corresponding transition matrix. In Corollary~\ref{c:entropy} we saw that $\htop(f)=\log \lambda_M$ and from Corollary~\ref{c:3} we know that $\lambda_M >1$. Given $\epsilon>0$ choose $\lambda>\lambda_M$ with $\log\lambda < \htop(f)+\epsilon$. In Proposition~\ref{p:sub} we identified a $\lambda$-subeigenvector $v$ which is deficient in only one coordinate. Summability of $v$ follows from the leo property, see Remark~\ref{r:5}(ii). Now applying Theorem~\ref{t:1}~(iii), Corollary~\ref{c:isometry}, and the fact that $\HD_L(G')$ of a countably affine graph is $1$, for the $\tau$-compatible metric $\varrho$ on $G$ we obtain
\[
	\HD_{\rho}(G)\cdot \log^+ \Lip_\varrho(f) = \HD_{L}(G')\cdot \log^+ \Lip_L(g)=\log\lambda. \qedhere
\]
\end{proof}

\begin{question}
Does Theorem~\ref{t:4} apply also to non-leo maps $f\in\CPM(G)$?
\end{question}

This seems to be a difficult question. Even in the case $G=[0,1]$ we do not know the answer. This is a different issue than the infimum of Lipschitz constants addressed in~\cite{BoRo17}.

%
%
%
%
%

\linespread{1}\selectfont 

\end{document}